\documentclass[12pt]{amsart}
\usepackage{a4wide,enumerate,xcolor,graphicx}
\usepackage{amsmath}
\usepackage{float}
\allowdisplaybreaks

\let\pa\partial
\let\na\nabla
\let\eps\varepsilon
\newcommand{\N}{{\mathbb N}}
\newcommand{\R}{{\mathbb R}}
\newcommand{\diver}{\operatorname{div}}

\numberwithin{equation}{section}
\newtheorem{theorem}{Theorem}[section]
\newtheorem{definition}[theorem]{Definition}
\newtheorem{lemma}[theorem]{Lemma}

\begin{document}

\title[Spectral-fractional Cahn--Hilliard cross-diffusion systems]{Global weak solutions to a fractional \\ Cahn--Hilliard cross-diffusion system \\ in lymphangiogenesis}

\author[A. J\"ungel]{Ansgar J\"ungel}
\address{Institute of Analysis and Scientific Computing, Technische Universit\"at, Wiedner Hauptstra\ss e 8--10, 1040 Wien, Austria}
\email{juengel@tuwien.ac.at}

\author[Y. Li]{Yue Li}
\address{Institute of Analysis and Scientific Computing, Technische Universit\"at, Wiedner Hauptstra\ss e 8-10, 1040 Wien, Austria}
\email{yue.li@asc.tuwien.ac.at}

\date{\today}

\thanks{The authors acknowledge partial support from
the Austrian Science Fund (FWF), grants P33010 and F65. This work has received funding from the European Research Council (ERC) under the European Union's Horizon 2020 research and innovation programme, ERC Advanced Grant no.~101018153.}

\begin{abstract}
A spectral-fractional Cahn--Hilliard cross-diffusion system, which describes the pre-patterning of lymphatic vessel morphology in collagen gels, is studied. The model consists of two higher-order quasilinear parabolic equations and describes the evolution of the fiber phase volume fraction and the solute concentration. The free energy consists of the nonconvex Flory--Huggins energy and a fractional gradient energy, modeling nonlocal long-range correlations. The existence of global weak solutions to this system in a bounded domain with no-flux boundary conditions is shown. The proof is based on a three-level approximation scheme, spectral-fractional calculus, and a priori estimates coming from the energy inequality.
\end{abstract}

\keywords{Cahn--Hilliard equation; cross-diffusion equations; spectral-fractional Laplacian; singular potential; free energy; weak solutions.}

\subjclass[2000]{35D30, 35K35, 35K65, 35K67, 92C37.}

\maketitle


\section{Introduction}

The pre-patterning of lymphatic vessel morphology in collagen gels was modeled by Roose and Fowler \cite{RoFo08} using evolution equations for the fiber phase volume fraction $\phi(x,t)$ and the concentration $c(x,t)$ of the solute. The model consists of a cross-diffusion system of Cahn--Hilliard type solved in a bounded domain with no-flux boundary conditions. It was shown in \cite{JuWa24} that a modified form of this model is thermodynamically consistent in the sense that the evolution is a gradient flow associated to the free energy
\begin{align*}
  E_{\rm loc}(\phi,c) = \int_{\R^d}\bigg(\frac12|\na\phi|^2 + f(\phi,c)
  \bigg)dx,
\end{align*}
which is the sum of the correlation energy density $\frac12|\na\phi|^2$ and the interaction-nutrient energy density $f(\phi,c)$. The equations model the separation of the phases $\phi=0$ and $\phi=1$ described by local short-range interactions. Nonlocal long-range interactions may occur as well \cite{Tan22}, leading to nonlocal free energies of the type
\begin{align*}
  E_{\rm nonloc}(\phi,c) = \int_{\R^d}\bigg(\frac14\int_{\R^d}\int_{\R^d}
  J(x-y)(\phi(x)-\phi(y))^2 dxdy + f(\phi,c)\bigg)dx,
\end{align*}
where $J$ is an integral kernel. We refer to \cite{ChFi00,GiLe97} for physical justifications of this approach. We wish to consider nonsmooth kernels close to the Laplacian and choose a fractional Laplacian operator. While there is a natural definition of the fractional Laplacian in the whole space, its definition in bounded domains may differ. In this paper, we use the spectral-fractional Laplacian definition. An existence analysis of the gradient-flow equations associated to the energy $E_{\rm loc}$ was given in \cite{JuLi24}. Our aim is to extend this analysis for the gradient-flow equations associated to a free energy in which the gradient term in $E_{\rm loc}$ is replaced by a fractional gradient term.

\subsection{Model equations}

More precisely, the diffusion system for lymphangiogenesis reads as
\begin{align}
  \pa_t\phi &= \diver\big(M(\phi)(\na\mu - c\na\pa_cf(\phi,c)) \big),
  \label{m1} \\
  \pa_tc &= -\diver\big(cM(\phi)(\na\mu - c\na\pa_c f(\phi,c)) \big)
  + \diver\big(ce^{-\phi}\na\pa_c f(\phi,c) \big), \label{m2} \\
  \mu &= (-\Delta)^s_\Omega \phi+\partial_\phi f(\phi,c)
  \quad\mbox{in }\Omega,\ t>0, \label{m3}
\end{align}
where $\Omega\subset\R^d$ ($d\ge 1$) is a bounded domain, $\pa_c=\pa/\pa c$, $\pa_\phi=\pa/\pa\phi$ are partial derivatives, and $(-\Delta)^s_\Omega$ is the spectral-fractional Laplacian on its domain $D((-\Delta)^s_\Omega)\subset L^2(\Omega)$ with parameter $0<s<1$, defined by the spectral decomposition
\begin{align*}
  (-\Delta)_\Omega^s\phi = \sum_{k=1}^\infty\lambda_k^s
  (\phi,e_k)_{L^2(\Omega)}e_k, \quad \phi\in D((-\Delta)^s_\Omega),
\end{align*}
where $(e_k,\lambda_k)$ are the eigenpairs of the Laplacian $-\Delta$ with homogeneous Neumann boundary conditions; see Section \ref{sec.fract} for details. One may use the Riesz fractional Laplacian as well, but the spectral-fractional Laplacian requires no information about $\phi$ in the exterior $\R^d\setminus\Omega$ \cite{LPG+20}.

The mobility $M(\phi)$ is assumed to be nondegenerate, i.e., it is strictly positive, and the Flory--Huggins energy density is given by
\begin{equation}\label{1.f}
  f(\phi,c) = \phi\log\phi + (1-\phi)\log(1-\phi) + \phi(1-\phi)
  + \frac{c^2}{2} + c(1-\phi) + \log 2.
\end{equation}
The first three terms describe the (nonconvex) Flory--Huggins energy \cite{Flo42,Hug41}, while the last three terms represent the nutrient energy \cite[(2.63)]{GKT22}. Equations \eqref{m1}--\eqref{m3} are supplemented by initial and no-flux boundary conditions,
\begin{align}
  \phi(0) = \phi_0, \quad c(0) = c_0 &\quad\mbox{in }\Omega, \label{ic} \\
  \na\phi\cdot\nu = c\na c\cdot\nu = \na\mu\cdot\nu = 0
  &\quad\mbox{on }\pa\Omega,\ t>0. \label{bc}
\end{align}
The goal of this paper is to prove the existence of global weak solutions.

\subsection{Key ideas}\label{sec.key}

The free energy of system \eqref{m1}--\eqref{m3} reads as
\begin{align*}
  E(\phi,c) = \int_\Omega\bigg(\frac12|(-\Delta)_\Omega^{s/2}\phi|^2
  + f(\phi,c)\bigg)dx.
\end{align*}
Thanks to the thermodynamically consistent modeling, the free energy is a Lyapunov functional and the following energy equality is satisfied:
\begin{align}\label{1.E}
  \frac{dE}{dt}(\phi,c)
  + \int_\Omega M(\phi)|\na\mu-c\na\pa_c f(\phi,c)|^2 dx
  + \int_\Omega ce^{-\phi}|\na\pa_c f(\phi,c)|^2 dx = 0.
\end{align}
Interestingly, the energy dissipation is the same as in the local model of \cite{JuLi24}. However, since $s<1$, we obtain less regularity than in \cite{JuLi24}, which makes the analysis more delicate. Further difficulties are the degeneracy at $c=0$ and the singularity of the potential $\pa_\phi f(\phi,c)$ at $\phi=0$ and $\phi=1$. Unlike \cite{JuLi24}, we assume a nondegenerate mobility, which allows us to overcome the mentioned difficulties.

The main task is the derivation of an a priori bound for $\mu$ in some Sobolev space. Clearly, in the degenerate case, a gradient bound cannot be expected. We first estimate as follows:
\begin{align*}
  \|\na\mu\|_{L^2(0,T;L^{4/3}(\Omega))}
  &\le C\|\na\mu-c\na\pa_cf(\phi,c)\|_{L^2(0,T;L^{4/3}(\Omega))}\\
  &\phantom{xx} +\|\sqrt{c}\|_{L^\infty(0,T;L^4(\Omega))}
  \|\sqrt{c}\na\pa_c f(\phi,c)\|_{L^2(0,T;L^2(\Omega))}.
\end{align*}
The right-hand side is bounded thanks to the energy equality \eqref{1.E} (here we use the positivity of $M(\phi)$). To estimate $\mu$ in $L^1(\Omega)$, we use equation \eqref{m3}:
\begin{align*}
  \int_\Omega \mu dx &= \int_\Omega
  \big((-\Delta)^s_\Omega\phi + \pa_\phi f(\phi,c)\big) dx
 = \int_\Omega \pa_\phi f(\phi,c)dx.
\end{align*}
The idea is to estimate the integral over $\pa_\phi f(\phi,c)$ in terms of $\na\mu$, using \eqref{m3} again and some techniques from \cite{GGG17}. Then
\begin{align*}
  \bigg\|\int_\Omega\mu dx\bigg\|_{L^2(0,T)} \le C\|\na\mu\|_{L^2(0,T;L^{4/3}(\Omega))} + C
  \le C,
\end{align*}
where $C>0$ is a constant independent of the solution. The Poincar\'e--Wirtinger inequality provides a bound for $\mu$ in $L^2(0,T;L^{4/3}(\Omega))$. This is the key estimate for deriving further bounds for $(-\Delta)_\Omega^s\phi$ in $L^2(0,T;L^2(\Omega))$ and for $c^{3/2}$ in $L^2(0,T;W^{1,4/3}(\Omega))$.

Compared to our previous work \cite{JuLi24}, we are able to derive a gradient bound for $\mu$. This was compensated in \cite{JuLi24} by the computation of an entropy inequality yielding an $L^2(\Omega)$ bound for $\Delta\phi$. Such an estimate cannot be expected in our case, and we obtain only a bound for $(-\Delta)_\Omega^s\phi$ in $L^2(\Omega)$ (recall that $s<1$). However, this estimate is sufficient even without the use of the entropy functional.

To derive a bound for $(-\Delta)_\Omega^s\phi$, we split the Flory--Huggins potential in the singular part  $f_1'(\phi)=\log\phi-\log(1-\phi)$ and the regular part $\pa_\phi f_2(\phi,c)=1-2\phi-c$ and multiply \eqref{m3} by $(-\Delta)_\Omega^s\phi$:
\begin{align}\label{1.aux}
  \int_\Omega|&(-\Delta)_\Omega^s\phi|^2 dx
  + \int_\Omega f'_1(\phi)(-\Delta)_\Omega^s\phi dx\\
  &= \int_\Omega\mu(-\Delta)_\Omega^s\phi dx
  + \int_\Omega(c-1+2\phi)(-\Delta)_\Omega^s\phi dx \nonumber\\
  &\le \frac12\|(-\Delta)_\Omega^s\phi\|_{L^2(\Omega)}^2
  + \|\mu\|_{L^2(\Omega)}^2 + \|c-1+2\phi\|_{L^2(\Omega)}^2. \nonumber
\end{align}
The first term on the right-hand side can be absorbed by the left-hand side. Since we already proved that $\mu$ is bounded in $W^{1,4/3}(\Omega)\hookrightarrow L^2(\Omega)$ (the embedding holds in up to four dimensions), the last two terms on the right-hand side are bounded. It was shown in \cite[Theorem 3.3]{STV19} that the second term on the left-hand side of \eqref{1.aux} is nonnegative, which can be interpreted as a weak form of the Stroock--Varapoulos inequality. However, this result only holds true if $f_1''$ is uniformly bounded, which is not true in our case. Therefore, we need to approximate $f_1'$ by some function $f'_{1,\delta}$ with some parameter $\delta>0$. With these arguments, we see that $(-\Delta)_\Omega^s\phi$ is bounded in $L^2(\Omega)$ for an approximation of $\phi$. These are the key estimates for the existence analysis. The limit $\delta\to 0$ can be performed by using the a priori estimates and compactness arguments.

\subsection{Main result}

We first detail our definition of weak solution. Set $\Omega_T=\Omega\times(0,T)$ and $\overline{\phi}_0=\mbox{meas}(\Omega)^{-1}\int_\Omega\phi_0 dx$.

\begin{definition}[Weak solution]\label{def}
Let $T>0$ be arbitrary and let $1/2\le s<1$. The function $(\phi,c)$ is called a {\em weak solution} to problem \eqref{m1}--\eqref{bc} on $[0,T]$ if $(\phi,c)$ satisfies $0<\phi<1$, $c\ge 0$ in $\Omega_T$,
\begin{align*}
  & \phi\in L^\infty(0,T;D((-\Delta)^{s/2}_\Omega))
  \cap L^2(0,T;D((-\Delta)^{s}_\Omega)), \quad
  \mu\in L^2(0,T;W^{1,4/3}(\Omega)), \\
  & c\in L^\infty(0,T;L^2(\Omega)),\quad
  c^{3/2}\in L^{4/3}(0,T;W^{1,4/3}(\Omega)), \\
  & M(\phi)(\nabla\mu-c\na\pa_cf(\phi,c))\in L^2(\Omega_T),
  \quad \sqrt{c}\na\pa_cf(\phi,c)\in L^2(\Omega_T), \\
  & \pa_t\phi\in L^2(0,T;H^{1}(\Omega)'), \quad
  \pa_t c\in L^{3/2}(0,T;W^{1,9}(\Omega)'),
\end{align*}
the initial conditions $\phi(0)=\phi_0$ in $L^2(\Omega)$,
$c(0)=c_0$ in the sense of $H^1(\Omega)'$, $(\phi,c)$ verifies the weak formulation
\begin{align*}
\int_0^T\langle\partial_t\phi,\psi_1\rangle_1dt
   &= -\int_0^T\int_{\Omega}M(\phi)(\nabla\mu-c\na\pa_cf(\phi,c))
  \cdot\nabla\psi_1 dxdt, \\
  \int_0^T\langle\partial_t c,\psi_2\rangle_2 dt
  &= \int_0^T\int_\Omega cM(\phi)(\nabla\mu-c\nabla\partial_cf(\phi,c))
  \cdot\nabla\psi_2dxdt\\
  &\phantom{xx}- \int^T_0\int_\Omega ce^{-\phi}
  \na\pa_cf(\phi,c)\cdot\na\psi_2dxdt
\end{align*}
for all $\psi_1\in L^2(0,T;H^1(\Omega))$, $\psi_2\in L^3(0,T;W^{1,9}(\Omega))$, and $\mu$ satisfies
$$
  \mu = (-\Delta)_\Omega^s \phi+\partial_\phi f(\phi,c)
  \quad\mbox{a.e.\ in }\Omega_T.
$$
Here, $\langle\cdot,\cdot\rangle_1$ is the dual product between $H^1(\Omega)'$ and $H^1(\Omega)$, and $\langle\cdot,\cdot\rangle_2$ is the dual product between $W^{1,9}(\Omega)'$ and $W^{1,9}(\Omega)$.
\end{definition}

Now we state our main result.

\begin{theorem}\label{main}
Assume that $\Omega\subset\R^d$ $(d\le 3)$ is a bounded domain with smooth boundary and let $1/2\leq s<1$, $T>0$. The mobility $M$ is continuous on $\R$ and satisfies 
\begin{align*}
  \gamma\leq M(s)\leq 1/\gamma\quad\mbox{for any }s\in\R
\end{align*} 
for positive constant $\gamma$. Let $\phi_0\in D((-\Delta)^{s/2}_\Omega)$, $c_0\in L^2(\Omega)$
satisfy $0<\overline{\phi}_0<1$, $0\leq\phi_0\leq 1$ and $c_0\ge 0$ in $\Omega$. Then problem \eqref{m1}--\eqref{bc} possesses a weak solution $(\phi,c)$ in $[0,T]$ in the sense of Definition \ref{def}.
\end{theorem}

Since we need to define $\na\phi$ in $L^2(\Omega)$ in the weak formulation, the embedding $\phi\in D((-\Delta)_\Omega^s) \hookrightarrow H^1(\Omega)$ is required, explaining the lower bound $s\ge 1/2$. The restriction of the space dimension basically comes from the embedding $W^{1,4/3}(\Omega)\hookrightarrow L^2(\Omega)$ which holds up to $d\le 4$. We believe that our results hold true if $d=4$ at the expense of the integrability of the unknowns. The condition $0<\overline{\phi}_0<1$ is used in Lemma \ref{lem.mudelta} to estimate an approximation of $f'_1(\phi)$ in $L^1(\Omega)$. Note that the initial phases may vanish in certain regions, but their integrals need to be positive.

Concerning the proof of Theorem \ref{main}, the complicated structure of system \eqref{m1}--\eqref{m3} makes a three-level approximation necessary. First, we remove the singularity in $f'(\phi,c)$ by truncation using the parameter $\delta>0$. Furthermore, we add some articial diffusion of order $\delta$ to the equations for $c$ and $\mu$, and we mollify the initial data. Second, we truncate the diffusion coefficient $c$ in the equation \eqref{m2} for the solute concentration by using a parameter $\eps>0$. Third, we solve the approximate problem by the Faedo--Galerkin method involving the dimension $N\in\N$ of the Faedo--Galerkin space.

The Faedo--Galerkin system is solved by applying Peano's theorem. Thanks to the approximate energy equality and the artificial diffusion, we can pass to the limits $N\to\infty$ and $\eps\to 0$. The delicate part of the proof is the limit $\delta\to 0$. For this, we derive additional estimates as described in Section \ref{sec.key} and apply a nonlinear version of the Aubin--Lions compactness lemma in the version of \cite{CJL14} to obtain the existence of a subsequence of an approximating sequence that converges to a weak solution $(\phi,c,\mu)$ to \eqref{m1}--\eqref{bc}.

\subsection{State of the art}

Before we prove Theorem \ref{main}, we briefly comment on the state of the art. While multi-phase models are intensively studied in the literature in the context of, for instance, tumor \cite{EGN21} and biofilm growth \cite{WaZh12} since several years, the two-phase modeling of lymphoangiogenesis is more recent \cite{RoFo08}. It was found in \cite{RoSw11} that a hexagonal lymphatic capillary network is optimal in terms of fluid drainage, confirmed by experiments in mouse tails and human skin. A hexagonal pre-structure was found in numerical simulations in two space dimensions \cite{JuWa24}.

The Cahn--Hilliard equation $\pa_t\phi - \Delta(-\Delta\phi + f'(\phi)) = 0$ with the (possibly singular) potential $f'(\phi)$ was introduced in \cite{CaHi58} to study phase separation in binary alloys. The local and nonlocal equations were derived in \cite{GiLe97} from a lattice gas evolving via Kawasaki exchange with respect to the Gibbs measure for a Hamiltonian. The localization limit was proved both in the nondegenerate \cite{DST21} and degenerate case \cite{ElSk23}.

The existence of one-dimensional solutions to the (local) Cahn--Hilliard equation was first proved in \cite{ElSo86} and later extended to several space dimensions in \cite{ElGa96}. The existence and uniqueness of solutions to the Cahn--Hilliard system strictly depend on the properties of the mobility $M(\phi)$ (being degenerate or nondegenerate) and the potential $f'(\phi)$ (being singular or regular). A mathematical difficulty is the fourth-order derivative, which excludes the use of comparison principles. A sufficient condition for the property $0\le\phi\le 1$ (if satisfied initially) is a degenerate mobility $M(\phi)$ \cite{Yin92} or a singular potential \cite{BlEl91}.

The study of fractional Cahn--Hilliard equations has recently received considerable attention. In \cite{AiMa17}, the fractional version
\begin{align*}
  \pa_t\phi + (-\Delta)_\Omega^s(-\Delta\phi + f'(\phi)) = 0
\end{align*}
was identified as a gradient flow of the free energy $E_{\rm loc}$ in the negative-order fractional Sobolev space $H^{-s}(\Omega)$ with $0<s<1$, and the existence of weak solutions was proved. On the other hand, the $H^{-1}(\Omega)$ gradient flow of the fractional free energy becomes
\begin{align*}
  \pa_t\phi - \Delta\big((-\Delta)_\Omega^s \phi + f'(\phi)\big) = 0,
\end{align*}
which was investigated in a bounded domain with periodic boundary conditions \cite{AiMa17a} and with no-flux boundary conditions \cite{GGG23}. Finally, the double-fractional Cahn--Hilliard equation
\begin{align*}
  \pa_t\phi + (-\Delta)_\Omega^s\big((-\Delta)_\Omega^\sigma \phi + f'(\phi)\big) = 0,
\end{align*}
with $0<s,\sigma<1$ and the singular integral representation of the fractional Laplacian was analyzed in \cite{ASS16,ASS19}. Instead of the fractional Laplacian, fractional powers of self-adjoint monotone operators were considered in \cite{CGS19}. We are not aware of papers on systems of fractional Cahn--Hilliard equations. Thus, up to our knowledge, the existence analysis of the spectral-fractional Cahn--Hilliard cross-diffusion system \eqref{m1}--\eqref{m3} is new.


\medskip
The paper is organized as follows. We recall the definition and some properties of the spectral-fractional Laplacian in Section \ref{sec.fract}. The approximate solution in the Faedo--Galerkin space and the limit $N\to\infty$ are proved in Section \ref{sec.approx}. The limit $\eps\to 0$ is performed in Section \ref{sec.eps}, while the more involved limit $\delta\to 0$ is shown in Section \ref{sec.delta}.


\section{The spectral-fractional Laplacian}\label{sec.fract}

We introduce the spectral-fractional Laplacian and recall some of its properties. We refer to \cite[Sec.~2.1]{GGG23} for details. The operator $(-\Delta)_\Omega$ denotes the Laplace operator $-\Delta$ with homogeneous Neumann boundary conditions with domain
\begin{align*}
  D((-\Delta)_\Omega) := \{u\in H^2(\Omega):\na u\cdot\nu=0\mbox{ on }
  \pa\Omega\}.
\end{align*}
By spectral theory, there exists a sequence of real nonnegative eigenvalues $(\lambda_k)_{k\in\N}$ satisfying $\lambda_0=0$, $\lambda_k\le\lambda_{k+1}$ for $k\in\N$, and $\lambda_k\to\infty$ as $k\to\infty$. The sequence of associated eigenfunctions $(e_k)_{k\in\N}\subset D((-\Delta)_\Omega)$ form an orthonormal basis of $L^2(\Omega)$. The eigenfunctions verify $(-\Delta)_\Omega e_k=\lambda_k e_k$ in $\Omega$ and $\overline{e}_k:=|\Omega|^{-1}\int_\Omega e_k dx=0$ for $k\in\N$. Any function $u\in D((-\Delta)_\Omega)$ can be represented by the series
\begin{align*}
  (-\Delta)_\Omega u = \sum_{k=1}^\infty
  \lambda_k(u,e_k)_{L^2(\Omega)}e_k.
\end{align*}
Based on this spectral decomposition, we define the (positive) fractional powers of order $s\in(0,1)$ by
\begin{align*}
  (-\Delta)^s_\Omega u = \sum_{k=1}^\infty
  \lambda_k^s(u,e_k)_{L^2(\Omega)}e_k \quad\mbox{for }
  u\in D((-\Delta)^s_\Omega),
\end{align*}
where the domain of $(-\Delta)^s_\Omega$ is given by
\begin{align*}
  D((-\Delta)^s_\Omega) := \bigg\{u\in L^2(\Omega):
  (-\Delta)^s_\Omega u\in L^2(\Omega)\mbox{ and }
  \int_\Omega(-\Delta)^s_\Omega u dx = 0\bigg\}.
\end{align*}
By definition, we have for $s,\sigma>0$ and $u\in D((-\Delta)^s_\Omega)\cap D((-\Delta)_\Omega^\sigma)$,
\begin{align*}
  \int_\Omega(-\Delta)_\Omega^s u(-\Delta)_\Omega^\sigma u dx
  = \int_\Omega|(-\Delta)_\Omega^{(s+\sigma)/2}u|^2 dx.
\end{align*}
With the norm of $u\in D((-\Delta)^s_\Omega)$,
\begin{align*}
  \|u\|_{D((-\Delta)^s_\Omega)}^2
  = \frac{1}{|\Omega|}\bigg(\int_\Omega udx\bigg)^2
  + \sum_{k=1}^\infty\lambda_k^{2s}|(u,e_k)_{L^2(\Omega)}|^2,
\end{align*}
the space $D((-\Delta)^s_\Omega)$ becomes a Banach space. This space can be related to the Sobolev--Slobodeckij space
\begin{align*}
  H^{2s}(\Omega) = \bigg\{u\in L^2(\Omega): \int_\Omega\int_\Omega
  \frac{|u(x)-u(y)|^2}{|x-y|^{d+2s}}dxdy<\infty\bigg\}, \quad
  0<s<1.
\end{align*}
Indeed, let
\begin{align}\label{2.Ds}
  D^s(\Omega) := D((-\Delta)_\Omega^{s/2})\quad\mbox{for }s>0.
\end{align}
Then it holds that \cite[(21)]{GGG23}
\begin{align*}
  D^s(\Omega) &= H^{s}(\Omega)
  &&\quad\mbox{if }0<s<\frac32, \\
  D^s(\Omega) &= \{u\in H^{3/2}(\Omega):\textstyle\int_\Omega
  \operatorname{dist}(x,\pa\Omega)^{-1}u(x)^2dx<\infty\}
  &&\quad\mbox{if }s=\frac32, \\
  D^s(\Omega) &= \{u\in H^{s}(\Omega):\na u\cdot\nu=0
  \mbox{ on }\pa\Omega\}
  &&\quad\mbox{if }\frac32<s<1.
\end{align*}
The identification with $H^{s}(\Omega)$ allows us to use standard Sobolev embedding theorems. In particular, for $p\ge 2$, the embedding
\begin{align*}
  D^s(\Omega) \hookrightarrow
  L^p(\Omega)\quad\mbox{for }s\ge \frac{d}{2}-\frac{d}{p},
\end{align*}
is continuous, and it is compact if $s>d/2-d/p$.

The following lemma, proved in \cite[Lemma A.1]{GGG23}, resembles the Stroock--Varopoulos inequality for the (singular integral representation) of the fractional Laplacian in $\R^d$ \cite[Theorem 3.3]{STV19}.

\begin{lemma}\label{lem.pos}
Let $0<s<1$ and let $F\in C^2(\R)$ be a convex function satisfying $F(1/2) = F'(1/2)=0$ and there exists $C>0$ such that $|F''(s)|\le C$ for any $s\in\R$. Then, for any $u\in D^{2s}(\Omega)$,
\begin{align*}
  \int_\Omega F'(u)(-\Delta)_\Omega^s u dx \ge 0.
\end{align*}
\end{lemma}


\section{Approximate solutions}\label{sec.approx}

We construct a regularized problem associated to \eqref{m1}--\eqref{bc} by approximating the singular part of the energy density and truncating the concentration-dependent diffusion coefficients. To this end, we split the free energy density \eqref{1.f} into a convex part $f_1$ and a nonconvex part $f_2$,
$$
  f_1(\phi) = \phi\log\phi + (1-\phi)\log(1-\phi) + \log 2, \quad
  f_2(\phi,c) = \phi(1-\phi) + \frac{c^2}{2} + c(1-\phi).
$$
As in \cite[(3.5)]{GGG17} (but different from \cite{JuLi24}), we define an approximation $f_{1,\delta}$ of $f_1$ on $\R$ to remove the singularities at $\phi=0$ and $\phi=1$. There exists a convex function $f_{1,\delta}:\R\to\R$ such that $f_{1,\delta}'$ is Lipschitz continuous on $\R$ with constant $1/\delta$ and $f_{1,\delta}(1/2)=f_{1,\delta}'(1/2)=0$. Additionally, $f_{1,\delta}\nearrow f_1$ in $[0,1]$, $|f'_{1,\delta}|\nearrow|f'_1|$ in $(0,1)$ as $\delta\to 0$, and for any $\delta^*>0$, there exists $C^*>0$ such that
\begin{align*}
  f_{1,\delta}(s) \ge \frac{s^2}{4\delta^*} - C^*
  \quad\mbox{for any }s\in\R\mbox{ and }0<\delta\le\delta^*.
\end{align*}
We set
$$
  f_\delta(\phi,c) = f_{1,\delta}(\phi) + f_2(\phi,c).
$$
Finally, we introduce the truncations
$$
  [\phi]_+^1 = \min\{1,\max\{0,\phi\}\}, \quad
  [c]_+^\eps = \min\{1/\eps,\max\{0,c\}\},
$$
where $0<\eps<1$. Then our approximate system reads as
\begin{align}
  \partial_t\phi &= \diver\big(M(\phi)(\nabla\mu-[c]_+^\eps
  \nabla\partial_cf_\delta(\phi,c)) \big),\label{3.approx1} \\
  \partial_t c &= -\diver\big([c]_+^\eps
  M(\phi)(\nabla\mu-[c]_+^\eps\nabla
  \partial_cf_\delta(\phi,c)) \big)
  + \diver\big([c]_+^\eps e^{-[\phi]^1_+}\nabla
  \partial_cf_\delta(\phi,c) \big)+\delta\Delta c, \label{3.approx2} \\
  \mu &= (-\Delta)_\Omega^s\phi + \partial_\phi f_\delta(\phi,c)
  - \delta\Delta\phi, \label{3.approx3}
\end{align}
with the initial and homogeneous Neumann boundary conditions
\begin{align}
  \phi(0)=\phi_{0,\delta},\quad c(0)=c_0
  & \quad\mbox{in }\Omega, \label{3.ic} \\
  \nabla\phi\cdot\nu = \nabla\mu\cdot\nu = \nabla c\cdot\nu = 0
  & \quad \mbox{on }(0,T)\times\partial\Omega. \label{3.bc}
\end{align}
Here, $\phi_{0,\delta}\subset H^1(\Omega)$ is an approximation of $\phi_0$ satisfying $0\le\phi_{0,\delta}\le 1$ in $\Omega$, $0<\overline{\phi}_{0,\delta}<1$, and for any $1\le p<\infty$,
\begin{align*}
  \phi_{0,\delta}\to \phi_0 \quad\mbox{strongly in }D^s(\Omega)
  \cap L^p(\Omega)\quad\mbox{as }\delta\to 0,
\end{align*}
recalling Definition \ref{2.Ds} of $D^s(\Omega)$. The remainder of this section is devoted to the solvability of the approximate problem \eqref{3.approx1}--\eqref{3.bc}.

\subsection{Faedo--Galerkin method}\label{sec.N}

Let $(e_k)_{k\in\N}$ be a complete orthonormal set of eigenfunctions of the Laplacian with homogeneous Neumann boundary conditions in $L^2(\Omega)$ and set $X_N=\operatorname{span}\{e_1,\ldots,e_N\}$ for $N\in\N$. Proceeding as in \cite[Sec.~2.1]{JuLi24}, there exists $T'>0$ and $(\phi_N,c_N,\mu_N)\in C^0([0,T'];X_N^3)$ solving
\begin{align}
  \int_{\Omega}\partial_t\phi_N edx
  &= -\int_\Omega M(\phi_N)\big(\nabla \mu_N-[c_N]_+^\eps
  \nabla \partial_c f_\delta(\phi_N,c_N)\big)\cdot\nabla edx,
  \label{a13} \\
  \int_\Omega \partial_t c_N edx
  &= \int_\Omega [c_N]_+^\eps M(\phi_N)\big(\na \mu_N-[c_N]_+^\eps\na \partial_c f_\delta(\phi_N,c_N)\big)\cdot
  \nabla edx \label{a14} \\
  &\phantom{xx} -\int_\Omega [c_N]_+^\eps e^{-[\phi_N]_+^1}
  \nabla\partial_c f_\delta(\phi_N,c_N)\cdot\nabla edx
  - \delta\int_\Omega \nabla c_N\cdot \nabla edx, \nonumber \\
  \int_\Omega \mu_N edx
  &= \int_{\Omega}(-\Delta)^{s/2}_\Omega\phi_N(-\Delta)^{s/2}_\Omega edx
  + \int_\Omega \partial_\phi f_\delta(\phi_N,c_N)edx
  +\delta\int_\Omega \nabla \phi_N\cdot\nabla edx,\label{a15}
\end{align}
for any $e\in X_N$, with initial conditions
\begin{align}\label{3.icN}
  \phi_N(0)=\sum_{k=0}^N(\phi_{0,\delta},e_k)_{L^2(\Omega)}e_k,\quad c_N(0)=\sum_{k=0}^N(c_0,e_k)_{L^2(\Omega)}e_k.
\end{align}
We wish to extend the solution globally in time. For this, we need bounds for $(\phi_N,c_N)$ in $X_N$. We start with the derivation of the approximate energy equality.

\begin{lemma}[Approximate energy equality]
Let $(\phi_N,c_N,\mu_N)\in C^0([0,T'];X_N^3)$ be the solution to \eqref{a13}--\eqref{3.icN}. Then
\begin{align}\label{3.energy}
  \frac{d}{dt}&\int_\Omega\bigg(\frac{1}{2}
  |(-\Delta)^{s/2}_\Omega\phi_N|^2
  + f_\delta(\phi_N,c_N)+\frac{\delta}{2}|\nabla\phi_N|^2 \bigg)dx \\
  &\phantom{xx}+ \int_\Omega M(\phi_N)\big|\nabla\mu_N
  - [c_N]_+^\eps\nabla\partial_c f_\delta(\phi_N,c_N)\big|^2 dx
  \nonumber \\
  &\phantom{xx}+ \int_\Omega [c_N]_+^\eps e^{-[\phi_N]^1_+}
  |\nabla\partial_c f_\delta(\phi_N,c_N)|^2dx \nonumber \\
  &= -\delta\int_\Omega\na c_N\cdot\na(c_N-\phi_N)dx. \nonumber
\end{align}
\end{lemma}

\begin{proof}
The proof is very similar to that one in \cite[Lemma 3.1]{JuLi24} with the exception that we have $(-\Delta)_\Omega^s$ instead of $-\Delta$. The idea is to choose the test functions $e=\mu_N$ in \eqref{a13}, $e=\pa_c f_\delta(\phi_N,c_N)$ in \eqref{a14}, $e=\pa_t\phi_N$ in \eqref{a15}, and adding the corresponding equations.
\end{proof}

\subsection{Uniform estimates in $N$}

We deduce from the approximate energy equality \eqref{3.energy} the following uniform estimates.

\begin{lemma}[Estimates for $\phi_N$ and $c_N$]\label{lem.est}
There exists a constant $C>0$ independent of $N$ such that
\begin{align}
  \big\|\nabla \mu_N-[c_N]_+^\eps\nabla
  \partial_cf_\delta(\phi_N,c_N)\big\|_{L^2(\Omega\times(0,T'))}
  &\leq C, \label{3.muN} \\
  \big\|([c_N]_+^\eps)^{1/2}\nabla\partial_c f_\delta(\phi_N,c_N)
  \big\|_{L^2(\Omega\times(0,T'))}
  &\leq C, \label{3.sqrtcN} \\
  \|c_N\|_{L^\infty(0,T';L^2(\Omega))}
  + \sqrt{\delta}\|\nabla c_N\|_{L^2(\Omega\times(0,T')}
  &\leq C, \label{3.nacN}\\
  \|\phi_N\|_{L^\infty(0,T';D^s(\Omega))}
  + \sqrt{\delta}\|\phi_N\|_{L^\infty(0,T';H^1(\Omega))} &\leq C. \label{3.phiN}
\end{align}
\end{lemma}

\begin{proof}
 The right-hand side of \eqref{3.energy} is estimated according to Young's inequality as
\begin{align*}
  -\delta\int_\Omega\na c_N\cdot\na(c_N-\phi_N)dx
  \le -\frac{\delta}{2}\int_\Omega|\nabla c_N|^2dx
  + \frac{\delta}{2}\int_\Omega |\nabla\phi_N|^2dx.
\end{align*}
The last term on the right-hand side can be estimated via Gronwall's inequality from the energy, while the other term is nonpositive. Next, by construction of $f_{1,\delta}$,
\begin{align*}
  f_\delta(\phi_N,c_N) &= f_{1,\delta}(\phi_N)
  + \phi_N(1-\phi_N) + \frac{c_N^2}{2} + c_N(1-\phi_N) \\
  &\ge \bigg(\frac{1}{4\delta^*}-C\bigg)\phi_N^2
  + \frac{c_N^2}{4} - C \ge C(\phi_N^2+c_N^2- 1),
\end{align*}
choosing $\delta^*>0$ sufficiently small. Then, taking into account the positive lower bound $M(\phi_N)\ge \gamma>0$, estimates \eqref{3.muN}--\eqref{3.phiN} follow from the energy equality, finishing the proof.
\end{proof}

The uniform bounds for $\phi_N$ and $c_N$ in $X_N$ imply the global existence of solutions for system \eqref{a13}--\eqref{3.icN}. Consequently, estimates \eqref{3.muN}--\eqref{3.phiN} are valid on the interval $[0,T]$ for any $T>0$. To take the limit $N\to\infty$, we also need an estimate for $\mu_N$.

\begin{lemma}[Estimate for $\mu_N$]\label{lem.muN}
There exists a constant $C>0$ independent of $N$ (but possibly depending on $\delta$ and $\eps$) such that
\begin{align}\label{f2}
  \|\mu_N\|_{L^2(0,T;H^1(\Omega))}\leq C.
\end{align}
\end{lemma}

\begin{proof}
It follows from \eqref{3.muN}--\eqref{3.nacN} that
\begin{align*}
  \|\nabla \mu_N\|_{L^2(\Omega_T)}
  &\le \big\|\nabla\mu_N-[c_N]_+^\eps\nabla
  \partial_c f_\delta(\phi_N,c_N)\big\|_{L^2(\Omega_T)} \\
  &+ \big\|([c_N]_+^\eps)^{1/2}\big\|_{L^\infty(\Omega_T)}
  \big\|([c_N]_+^\eps)^{1/2}\nabla \partial_c
  f_\delta(\phi_N,c_N)\big\|_{L^2(\Omega_T)} \leq C(\eps). \nonumber
\end{align*}
It remains to derive an $L^2(\Omega)$ bound for $\mu_N$. By the property $f_{1,\delta}(1/2)=0$ and the Lipschitz continuity of $f'_{1,\delta}$,
\begin{align}\label{3.aux}
  |f'_{1,\delta}(\phi_N)|^2
  = |f'_{1,\delta}(\phi_N)-f'_{1,\delta}(1/2)|^2
  \leq C(\delta)|\phi_N-1/2|^2\leq C|\phi_N|^2 + C.
\end{align}
Then, taking the test function $e_1=1$ (which is the eigenvalue of $(-\Delta)_\Omega^s$ with associated eigenvalue $\lambda_1=0$) in \eqref{a15},
\begin{align}\label{3.aux2}
  \bigg|\int_\Omega\mu_N dx\bigg|
  &= \bigg|\int_\Omega\big(f'_{1,\delta}(\phi_N)
  - 2\phi_N- c_N + 1\big)dx\bigg| \\
  &\le \|f'_{1,\delta}(\phi_N)\|_{L^1(\Omega)}
  + C\|\phi_N\|_{L^1(\Omega)} + \|c_N\|_{L^1(\Omega)} + C
  \nonumber \\
  &\le C\|\phi_N\|_{L^2(\Omega)} + C \le C, \nonumber
\end{align}
where we used \eqref{3.aux} and \eqref{3.nacN}--\eqref{3.phiN}. It follows from the Poincar\'e--Wirtinger inequality that
\begin{align*}
  \|\mu_N\|_{L^2(\Omega_T)}
  \leq \|\mu_N-\overline{\mu}_N\|_{L^2(\Omega_T)}
  + \|\overline{\mu}_N\|_{L^2(\Omega_T)}
  \leq C\|\nabla\mu_N\|_{L^2(\Omega_T)} + C \leq C.
\end{align*}
This finishes the proof.
\end{proof}

We derive further estimates for $\phi_N$.

\begin{lemma}[Estimates for $\phi_N$]\label{lem.phiN}
There exists a constant $C>0$ independent of $N$ (but possibly depending on $\delta$ and $\eps$) such that
\begin{align}\label{f5}
  \|(-\Delta)^s_\Omega \phi_N\|_{L^2(\Omega_T)}
  + \sqrt{\delta}\|\Delta\phi_N\|_{L^2(\Omega_T)}\leq C.
\end{align}
\end{lemma}

\begin{proof}
Using the test function $e=(-\Delta)^s_\Omega\phi_N$ in \eqref{a15}, we obtain
\begin{align*}
  \int_\Omega |&(-\Delta)^s_\Omega\phi_N|^2dx
  + \delta\int_\Omega |(-\Delta)^{(1+s)/2}_\Omega\phi_N|^2dx\\
  &=\int_\Omega \mu_N(-\Delta)^s_\Omega\phi_Ndx
  -\int_\Omega\partial_\phi f_\delta(\phi_N,c_N)
  (-\Delta)^s_\Omega\phi_Ndx\\
  &\leq \frac{1}{2}\int_\Omega|(-\Delta)^s_\Omega \phi_N|^2dx
  + \int_\Omega|\mu_N|^2dx
  + \int_\Omega |\partial_\phi f_\delta(\phi_N,c_N)|^2dx.
\end{align*}
We derive similarly as in \eqref{3.aux2} a uniform $L^2(\Omega_T)$ bound for $\pa_\phi f_\delta(\phi_N,c_N)$. Then, together with the bound \eqref{f2} for $\mu_N$, we infer the first bound in \eqref{f5}. Similarly, the test function $e=\Delta\phi_N$ in \eqref{a15} yields the second bound in \eqref{f5}.
\end{proof}

The following lemma gives estimates for the time derivatives.

\begin{lemma}[Estimates for the time derivatives]\label{lem.time}
There exists a constant $C>0$ independent of $N$ such that
\begin{align}\label{a29}
  \|\partial_t\phi_N\|_{L^2(0,T;H^{1}(\Omega)')}
  + \|\partial_tc_N\|_{L^{2}(0,T;H^{1}(\Omega)')} \leq C.
\end{align}
\end{lemma}

\begin{proof}
The proof follows from Lemmas \ref{lem.est}--\ref{lem.phiN} in a similar way as the proof of Lemma 3.4 of \cite{JuLi24}.
\end{proof}

\subsection{The limit $N\to\infty$}

The uniform estimates from Lemmas \ref{lem.est}--\ref{lem.time} imply the existence of a subsequence of $(\phi_N,c_N,\mu_N)$, which is not relabeled, such that, as $N\to\infty$,
\begin{align}
  & \phi_N\rightharpoonup^*\phi \quad\mbox{weakly* in }
  L^\infty(0,T;D^s(\Omega)\cap H^1(\Omega))
  \cap L^2(0,T;D^{2s}(\Omega)), \label{f6} \\
  & c_N\rightharpoonup c, \quad \mu_N\rightharpoonup \mu
  \quad\mbox{weakly in } L^2(0,T;H^1(\Omega)), \label{10} \\
  & \partial_t\phi_N\rightharpoonup \partial_t\phi, \quad
  \partial_tc_N\rightharpoonup \partial_tc
  \quad\mbox{weakly in }L^2(0,T;H^{1}(\Omega)'). \nonumber 
\end{align}
Moreover, by the Aubin--Lions lemma, up to subsequences,
\begin{align}
  \phi_N\to \phi &\quad\mbox{strongly in }
  C([0,T];L^2(\Omega)), \label{11} \\
  c_N\to c &\quad\mbox{strongly in }
  L^2(\Omega_T)\cap C([0,T];H^1(\Omega)'). \label{12}
\end{align}
In particular, again up to subsequences, $\phi_N\to\phi$ and $c_N\to c$ a.e.\ in $\Omega_T$. Then, since $M$, $\exp[\cdot]_+^1$, and $[\cdot]_+^\eps$ are bounded continuous functions, for any $1\le p<\infty$,
\begin{align*}
  M(\phi_N)\to M(\phi),\quad
  e^{-[\phi_N]_+^1}\to e^{-[\phi]_+^1},\quad
  [c_N]_+^\eps\to [c]_+^\eps
  \quad\mbox{strongly in } L^p(\Omega_T). 
\end{align*}
These convergences as well as the uniform bound \eqref{3.muN} imply that
\begin{align*}
  M(&\phi_N)\big(\nabla \mu_N-[c_N]_+^\eps
  (\nabla c_N-\nabla \phi_N) \big) \\
  &\to M(\phi)\big(\nabla \mu-[c]_+^\eps
  (\nabla c-\nabla \phi) \big)\quad\mbox{weakly in } L^2(\Omega_T). \nonumber
\end{align*}
We deduce from the a.e.\ convergences of $(\phi_N)$ and $(c_N)$ that $\pa_\phi f_\delta(\phi_N,c_N)\rightharpoonup\pa_\phi f_\delta(\phi,c)$ a.e.\ in $\Omega_T$ and then, because of the uniform bound for $\pa_\phi f_\delta(\phi_N,c_N)$ in $L^2(\Omega_T)$,
\begin{align}\label{a12}
  \partial_\phi f_\delta(\phi_N,c_N)\rightharpoonup \partial_\phi f_\delta(\phi,c) \quad\mbox{weakly in }L^2(\Omega_T).
\end{align}

The convergence results \eqref{f6}--\eqref{a12} allow us to perform the limit $N\to\infty$ in system \eqref{a13}--\eqref{a15} to infer that
the limit function $(\phi,c)$ solves system \eqref{3.approx1}--\eqref{3.approx2} in the sense of $L^2(0,T;H^1(\Omega)')$, and the limit function $\mu$ satisfies \eqref{3.approx3} a.e.\ in $\Omega_T$. It follows from \eqref{11} and $\phi_N(0)\to\phi_{0,\delta}$ strongly in $L^2(\Omega)$ that $\phi(0)=\phi_{0,\delta}$ in $\Omega$. Furthermore, we deduce from \eqref{12} that $\langle c_N(0),\xi\rangle_1\rightarrow\langle c(0),\xi\rangle_1$ for any $\xi\in H^1(\Omega)$,
recalling that $\langle\cdot,\cdot\rangle_1$ is the dual product between $H^1(\Omega)'$ and $H^1(\Omega)$. Then $c_N(0)\to c_0$ strongly in $H^1(\Omega)'$ implies that $c(0)=c_0$ in the sense of $H^1(\Omega)'$.

We wish to pass to the limit $N\to\infty$ in the energy equality \eqref{3.energy}. This is possible by the previous convergences and the weak lower semicontinuity of convex functions except of right-hand side of \eqref{3.energy}. Because of the $L^2(0,T;H^1(\Omega))$ bound for $\na\phi_N$ from \eqref{f5} and the $L^2(0,T;H^2(\Omega)')$ bound for $\pa_t\na\phi_N$ from \eqref{a29}, the Aubin--Lions lemma implies that (up to a subsequence) $\na\phi_N\to\na\phi$ strongly in $L^2(\Omega_T)$. Together with the weak convergence of $\na c_N$ in $L^2(\Omega_T)$ from \eqref{10}, we obtain
\begin{align*}
  \int_0^\tau\int_\Omega \nabla c_N\cdot\nabla\phi_Ndxdt
  \to \int_0^\tau\int_\Omega \nabla c\cdot\nabla\phi dxdt
  \quad\mbox{for }\tau>0.
\end{align*}
Therefore, the limit function $(\phi,c,\mu)$ satisfies that, for any $\tau\in (0,T)$,
\begin{align}\label{15}
  \int_\Omega&\bigg(\frac{1}{2}|(-\Delta)^{s/2}_\Omega\phi|^2
  + f_\delta(\phi,c) +\frac{\delta}{2}|\nabla\phi|^2\bigg)(x,\tau)dx\\
  &\phantom{xx}+ \int_0^\tau\int_\Omega M(\phi)
  \big|\nabla \mu - [c]_+^{\eps}\nabla\partial_c
  f_\delta(\phi,c)\big|^2dxdt \nonumber\\
  &\phantom{xx} + \int_0^\tau\int_\Omega [c]_+^\eps
  e^{-[\phi]^1_+}|\nabla\partial_c
  f_\delta(\phi,c)|^2dxdt
  + \delta\int_0^\tau\int_\Omega |\nabla c|^2
  dxdt \nonumber \\
  &\le \int_\Omega\bigg(\frac{1}{2}|(-\Delta)^{s/2}_\Omega
  \phi_{0,\delta}|^2 + f_\delta(\phi_{0,\delta},c_0) +\frac{\delta}{2}|\nabla\phi_{0,\delta}|^2\bigg)dx\nonumber\\
  &\phantom{xx}+ \delta\int_0^\tau\int_\Omega \nabla c
  \cdot\nabla \phi dxdt. \nonumber
\end{align}


\section{The limit $\eps\to 0$}\label{sec.eps}

Our goal of this section is to perform the limit $\eps\to 0$ in the weak formulation of \eqref{3.approx1}--\eqref{3.approx3} to remove the truncation. We denote by $(\phi_\eps,c_\eps,\mu_\eps)$  the solution constructed in the previous section.

First, we notice that the test function $c_\eps^-:=-\min\{0,c_\eps\}$ in the weak formulation of \eqref{3.approx2} shows that $c_\eps\ge 0$ in $\Omega_T$ (since $[c_\eps]_+^\eps c_\eps^- = 0$). Hence, we can replace the truncation $|c_\eps]_+^\eps$ by $[c_\eps]^\eps:=\min\{1/\eps,c_\eps\}$. By the same arguments as those used in the proofs of Lemma \ref{lem.est}, we infer from the energy inequality \eqref{15} that there exists a constant $C>0$ independent of $\eps$ such that
\begin{align}
  \big\|\nabla \mu_\eps - [c_\eps]^\eps\nabla\partial_cf_\delta
  (\phi_\eps,c_\eps)\big\|_{L^2(\Omega_T)}
  &\leq C, \label{ad5} \\
  \big\|\sqrt{[c_\eps]^\eps}\nabla\partial_c f_\delta(\phi_\eps,c_\eps)
  \big\|_{L^2(\Omega_T)}
  &\leq C, \label{ad6} \\
  \|c_\eps\|_{L^\infty(0,T;L^2(\Omega))}
  + \sqrt{\delta}\|\nabla c_\eps\|_{L^2(\Omega_T)}
  &\leq C, \label{ad7} \\
  \|\phi_\eps\|_{L^\infty(0,T;D^s(\Omega))}
  + \sqrt{\delta}\|\nabla\phi_\eps\|_{
  L^\infty(0,T;L^2(\Omega))}&\leq C. \nonumber 
\end{align}
In view of \eqref{ad6}--\eqref{ad7}, the sequence $([c_\eps]^\eps\na\pa_c f_\delta(\phi_\eps,c_\eps))$ is bounded in $L^2(0,T;$ $L^{4/3}(\Omega))$. The arguments in the proof of Lemma \ref{lem.muN} show that
\begin{align*}
  \|\mu_\eps\|_{L^2(0,T;W^{1,4/3}(\Omega))}\le C.
\end{align*}
Furthermore, we deduce from \eqref{ad5} that
\begin{align*}
  \|\partial_t\phi_\eps\|_{L^2(0,T;H^{1}(\Omega)')}
  \leq C. 
\end{align*}

The arguments for the limit $\eps\to 0$ are similar to those given in Section \ref{sec.approx}, except for the strong compactness of $[c_\eps]^\eps$. We only focus on this term. We have, by interpolation,
\begin{align}\label{f40}
  \|c_\eps\|_{L^4(0,T;L^3(\Omega))}
  \le \|c_\eps\|_{L^\infty(0,T;L^2(\Omega))}^{1/2}
  \|c_\eps\|_{L^2(0,T;L^6(\Omega))}^{1/2}\le C,
\end{align}
where we used the continuous embedding $H^1(\Omega)\hookrightarrow L^6(\Omega)$ for $d\le 3$. It follows for any $\psi\in L^4(0,T;W^{1,6}(\Omega))$ that
\begin{align*}
  \bigg|\int_0^T\int_\Omega \pa_t c_\eps\psi dxdt \bigg|
  &\le \|[c_\eps]^\eps\|_{L^4(0,T;L^3(\Omega))}
  \|M(\phi_\eps)\|_{L^\infty(\Omega_T)} \\
  &\phantom{xxxx}\times\big\|\na\mu_\eps-[c_\eps]^\eps \na\pa_c f_\delta(\phi_\eps,c_\eps)\big\|_{L^2(\Omega_T)}
  \|\na\psi\|_{L^4(0,T;L^6(\Omega))}\\
  &\phantom{xx} +\big\|\sqrt{[c_\eps]^\eps}\big\|_{L^4(0,T;L^3(\Omega))}
  \big\|e^{-[\phi_\eps]_+^1}\big\|_{L^\infty(\Omega_T)} \\
  &\phantom{xxxx}\times\big\|\sqrt{[c_\eps]^\eps}
  \na\pa_c f_\delta(\phi_\eps,c_\eps)\big\|_{L^2(\Omega_T)}
  \|\na\psi\|_{L^4(0,T;L^6(\Omega))}\\
  &\phantom{xx} +\delta\|\na c_\eps\|_{L^2(\Omega_T)}
  \|\na\psi\|_{L^2(\Omega_T)}
  \le C\|\psi\|_{L^4(0,T;W^{1,6}(\Omega))},
\end{align*}
which implies that
\begin{align*}
  \|\pa_t c_\eps\|_{L^{4/3}(0,T;W^{1,6}(\Omega)')}\le C.
\end{align*}
Together with the gradient bound \eqref{ad7} for $c_\eps$, the Aubin--Lions lemma yields the existence of a subsequence (not relabeled) such that $c_\eps\to c$ strongly in $C^0([0,T];H^1(\Omega)')\cap L^2(\Omega_T)$ and a.e.\ in $\Omega_T$ as $\eps\to 0$. Since
\begin{align*}
  \|[c_\eps]^\eps-c_\eps\|_{L^1(\Omega_T)}
  &= \int_0^T\int_{\{c_\eps\ge 1/\eps \}} (c_\eps-1/\eps)dxdt
  \le \int_0^T\int_{\{c_\eps\ge 1/\eps \}} c_\eps dxdt\\
  &\le \eps \int_0^T\int_\Omega c_\eps^2 dxdt\le C\eps\to 0,
\end{align*}
we have $[c_\eps]^\eps\to c$ a.e.\ in $\Omega_T$ and hence, due to \eqref{f40}, for any $p\in[1,4)$ and $q\in[1,3)$,
\begin{align*}
  [c_\eps]^\eps \to c \quad {\mbox{strongly in }}L^p(0,T;L^q(\Omega)).
\end{align*}
Proceeding as in Lemma \ref{lem.phiN}, we obtain uniform bounds for $\phi_\eps$ in $L^2(0,T;D^{2s}(\Omega))$ and for $\sqrt{\delta}\Delta\phi_\eps$ in $L^2(\Omega_T)$. Thus, in the limit $\eps\to 0$, the limit $\phi$ of $\phi_\eps$ satisfies
\begin{align}\label{3.phi}
  \phi\in L^2(0,T;D^{2s}(\Omega)), \quad
  \Delta\phi\in L^2(\Omega_T).
\end{align}

Now, we can pass to the limit $\eps\to 0$ in the weak formulation of \eqref{3.approx1}--\eqref{3.approx3} to deduce that the triplet $(\phi,c,\mu)$ (the limit of $(\phi_\eps,c_\eps,\mu_\eps)$) is a weak solution to
\begin{align}
  \partial_t\phi &= \diver\big(M(\phi)(\nabla\mu
  - c\nabla\partial_cf_\delta(\phi,c)) \big), \label{ad16} \\
  \partial_t c &= -\diver\big(c M(\phi) (\nabla\mu-c\nabla\partial_cf_\delta(\phi,c)) \big)
  + \diver\big(c e^{-[\phi]^1_+}\nabla\partial_cf_\delta(\phi,c) \big)
  + \delta\Delta c, \label{ad17} \\
  \mu &= (-\Delta)^s_\Omega\phi + \partial_\phi f_\delta(\phi,c)-\delta\Delta\phi, \label{ad18}
\end{align}
with the initial and boundary conditions \eqref{3.ic}--\eqref{3.bc}. We observe that, in view of \eqref{3.phi}, equation \eqref{ad18} holds a.e.\ in $\Omega_T$.  Furthermore, we deduce from \eqref{15} that the limit function satisfies the energy inequality
\begin{align}\label{ad15}
  \int_\Omega & \bigg(\frac{1}{2}|(-\Delta)^{s/2}_\Omega\phi|^2+ f_\delta(\phi,c)+\frac{\delta}{2}|\nabla\phi|^2
   \bigg)(x,\tau) dx\\
  &\phantom{xx}+ \int_0^\tau\int_\Omega M(\phi)\big|\nabla \mu
  - c\nabla\partial_c f_\delta(\phi,c)\big|^2 dxdt \nonumber\\
  &\phantom{xx} + \int_0^\tau\int_\Omega ce^{-[\phi]^1_+}
  |\nabla\partial_c f_\delta(\phi,c)|^2 dxdt
  + \delta\int_0^\tau\int_\Omega |\nabla c|^2 dxdt \nonumber \\
  &\le \int_\Omega\bigg(\frac{1}{2}
  |(-\Delta)^{s/2}_\Omega\phi_{0,\delta}|^2
  + f_\delta(\phi_{0,\delta},c_0)
  + \frac{\delta}{2}|\nabla\phi_{0,\delta}|^2
   \bigg)dx\nonumber\\
  &\phantom{xx}+\delta\int_0^\tau\int_\Omega \nabla c\cdot\nabla \phi dxdt. \nonumber
\end{align}


\section{The limit $\delta\to 0$}\label{sec.delta}

We perform the limit $\delta\to 0$ in system \eqref{ad16}--\eqref{ad18} to complete the proof of Theorem \ref{main}. Let $(\phi_\delta,c_\delta,\mu_\delta)$ be the solution to \eqref{ad16}--\eqref{ad18} with initial and boundary conditions \eqref{3.ic}--\eqref{3.bc}. We first derive estimates uniform in $\delta$, conclude convergence from compactness arguments, and pass to the limit $\delta\to 0$ in equations \eqref{ad16}--\eqref{ad18}.

\subsection{Estimates uniform in $\delta$}

Based on the energy inequality \eqref{ad15} and proceeding similarly as in Section \ref{sec.N}, we find that there exists a constant $C>0$ independent of $\delta$ such that
\begin{align}
  \|\phi_\delta\|_{L^\infty(0,T;D^s(\Omega))}
  +\sqrt{\delta}\|\phi_\delta\|_{L^\infty(0,T;H^1(\Omega))}
  &\leq C,\label{f11}\\
  \big\|\nabla \mu_\delta - c_\delta\nabla\partial_cf_\delta
  (\phi_\delta,c_\delta)\big\|_{L^2(\Omega_T)}
  &\leq C, \label{16} \\
  \|c_\delta\|_{L^\infty(0,T;L^2(\Omega))}
  + \sqrt{\delta}\|\nabla c_\delta\|_{L^2(\Omega_T)}
  &\leq C, \label{17} \\
  \|\sqrt{c_\delta}\nabla\partial_cf_\delta
  (\phi_\delta,c_\delta)\|_{L^2(\Omega_T)}
  +\|\nabla\mu_\delta\|_{L^2(0,T;L^{4/3}(\Omega))}
  &\leq C. \label{18}
\end{align}

In contrast to Section \ref{sec.eps}, the function $f'_{1,\delta}$ is not Lipschitz continuous with a constant independent of $\delta$. Thus, we need to find another way to obtain a uniform estimate for $\mu_\delta$. Note that this part of the proof is substantially different from \cite{JuLi24}, since in that paper, the mobility $M(\phi_\delta)$ is degenerate which excludes gradient bounds for $\mu_\delta$.

\begin{lemma}[Estimate for $\mu_\delta$]\label{lem.mudelta}
There exists a constant $C>0$ independent of $\delta$ such that
\begin{align}
  \|\mu_\delta\|_{L^2(0,T;W^{1,4/3}(\Omega))}\leq C.\label{f15}
\end{align}
\end{lemma}

\begin{proof}
We use the test function $\phi_\delta-\overline{\phi}_{0,\delta}\in L^2(0,T;H^1(\Omega))$ in the weak formulation of \eqref{ad18} to find that
\begin{align*}
  \int_\Omega& (-\Delta)^{s/2}_\Omega \phi_\delta
  (-\Delta)^{s/2}_\Omega(\phi_\delta-\overline{\phi}_{0,\delta})dx
  + \int_\Omega f'_{1,\delta}(\phi_\delta)
  (\phi_\delta-\overline{\phi}_{0,\delta})dx
  + \delta\int_\Omega |\nabla\phi_\delta|^2dx\\
  &= \int_\Omega\mu_\delta(\phi_\delta-\overline{\phi}_{0,\delta})dx
  - \int_\Omega \pa_\phi f_2(\phi_\delta,c_\delta)
  (\phi_\delta-\overline{\phi}_{0,\delta})dx \\
  &= \int_\Omega\mu_\delta(\phi_\delta-\overline{\phi}_{0,\delta})dx
  + \int_\Omega c_\delta(\phi_\delta-\overline{\phi}_{0,\delta})dx
  + \int_\Omega (2\phi_\delta-1)
  (\phi_\delta-\overline{\phi}_{0,\delta})dx.
\end{align*}
The first and last terms on the left-hand side are nonnegative. Therefore,
\begin{align}\label{f41}
  \int_\Omega f'_{1,\delta}(\phi_\delta)
  (\phi_\delta-\overline{\phi}_{0,\delta})dx
  &\le \int_\Omega\mu_\delta(\phi_\delta-\overline{\phi}_{0,\delta})dx
  + \int_\Omega c_\delta(\phi_\delta-\overline{\phi}_{0,\delta})dx \\
  &\phantom{xx}+ \int_\Omega (2\phi_\delta-1)
  (\phi_\delta-\overline{\phi}_{0,\delta})dx.\nonumber
\end{align}
We consider the first term on the right-hand side. Bound \eqref{f11} for $\phi_\delta$ and the continuous embedding $D^{s}(\Omega)\hookrightarrow L^2(\Omega)$ for any $s>0$ imply that $(\phi_\delta)$ is bounded in $L^\infty(0,T;L^2(\Omega))$. Consequently, taking into account the Poincar\'e--Wirtinger inequality \cite[Theorems 8.11--8.12]{LiLo01} (here, we need $d\le 4$ to guarantee the embedding $W^{1,4/3}(\Omega)\hookrightarrow L^2(\Omega)$) and the bound \eqref{18} for $\na\mu_\delta$,
\begin{align*}
  \int_\Omega\mu_\delta(\phi_\delta-\overline{\phi}_{0,\delta})dx
  &=\int_\Omega (\mu_\delta-\bar\mu_\delta)
  (\phi_\delta-\overline{\phi}_{0,\delta})dx\\
  &\leq \|\mu_\delta-\overline{\mu}_\delta\|_{L^{2}(\Omega)}
  \|\phi_\delta-\overline{\phi}_{0,\delta}\|_{L^2(\Omega)}
  \leq C\|\nabla\mu_\delta\|_{L^{4/3}(\Omega)}.\nonumber
\end{align*}
The remaining three terms on the right-hand side of \eqref{f41} are bounded by the $L^2(\Omega_T)$ norms of $\phi_\delta$ and $c_\delta$, which are bounded by \eqref{f11} and \eqref{17}. We conclude from \eqref{f41} that
\begin{align}\label{3.f1}
  \int_\Omega f'_{1,\delta}(\phi_\delta)
  (\phi_\delta-\bar\phi_{0\delta})dx
  \leq C\|\nabla\mu_\delta\|_{L^{4/3}(\Omega)} + C.
\end{align}

It is proved in \cite[p.~5270]{GGG17} that there exists $C>0$ independent of $\delta$ such that
\begin{align}\label{3.f2}
  \int_\Omega |f_{1,\delta}'(\phi_\delta)|dx
  \le \int_\Omega f_{1,\delta}'(\phi_\delta)(\phi_\delta-\overline{\phi}_{0,\delta})
  dx + C.
\end{align}
For the convenience of the reader, we recall the proof. Let $m_1,m_2\in(0,1)$ be such that $m_1\le 1/2\le m_2$ and $m_1<\overline{\phi}_{0,\delta}<m_2$. We set
\begin{align*}
  \eta_0=\min\{\overline{\phi}_{0,\delta}-m_1,
  m_2-\overline{\phi}_{0,\delta}\}, \quad
  \eta_1=\max\{\overline{\phi}_{0,\delta}-m_1,
  m_2-\overline{\phi}_{0,\delta}\},
\end{align*}
and we introduce the sets
\begin{align*}
  \Omega_0 = \{m_1\le\phi_\delta\le m_2\}, \quad
  \Omega_1 = \{\phi_\delta<m_1\}, \quad \Omega_2 = \{\phi_\delta>m_2\}.
\end{align*}
Since $f_{1,\delta}$ is convex, $f'_{1,\delta}$ is nondecreasing. Then it follows from $f_{1,\delta}'(1/2)=0$ that $f'_{1,\delta}(s)\le 0$ for $s\in(0,1/2)$ and $f'_{1,\delta}(s)\ge 0$ for $s\in(1/2,1)$. Consequently, $f'_{1,\delta}\le 0$ in $\Omega_1$ and $f'_{1,\delta}\ge 0$ in $\Omega_2$, and we infer that
\begin{align*}
  \eta_0\int_{\Omega_1}|f_{1,\delta}'(\phi_\delta)|dx
  &\le -\int_{\Omega_1}(\overline{\phi}_{0,\delta}-m_1)
  f'_{1,\delta}(\phi_\delta)dx
  \le \int_{\Omega_1}(\phi_\delta-\overline{\phi}_{0,\delta})
  f'_{1,\delta}(\phi_\delta)dx, \\
  \eta_0\int_{\Omega_2}|f_{1,\delta}'(\phi_\delta)|dx
  &\le \int_{\Omega_2}(m_2-\overline{\phi}_{0,\delta})
  f'_{1,\delta}(\phi_\delta)dx
  \le \int_{\Omega_2}(\phi_\delta-\overline{\phi}_{0,\delta})
  f'_{1,\delta}(\phi_\delta)dx.
\end{align*}
This yields
\begin{align*}
  \eta_0&\int_\Omega |f_{1,\delta}'(\phi_\delta)|dx
  = \eta_0\int_{\Omega_0}|f_{1,\delta}'(\phi_\delta)|dx
  + \eta_0\int_{\Omega_1}|f_{1,\delta}'(\phi_\delta)|dx
  + \eta_0\int_{\Omega_2}|f_{1,\delta}'(\phi_\delta)|dx \\
  &\le \eta_0\int_{\Omega_0}|f_{1,\delta}'(\phi_\delta)|dx
  + \int_{\Omega_1}(\phi_\delta-\overline{\phi}_{0,\delta})
  f'_{1,\delta}(\phi_\delta)dx
  + \int_{\Omega_2}(\phi_\delta-\overline{\phi}_{0,\delta})
  f'_{1,\delta}(\phi_\delta)dx \\
  &\le \eta_0\int_{\Omega_0}|f_{1,\delta}'(\phi_\delta)|dx
  - \int_{\Omega_0}(\phi_\delta-\overline{\phi}_{0,\delta})
  f'_{1,\delta}(\phi_\delta)dx
  + \int_{\Omega}(\phi_\delta-\overline{\phi}_{0,\delta})
  f'_{1,\delta}(\phi_\delta)dx \\
  &\le (\eta_0+\eta_1)\int_{\Omega_0}|f_{1,\delta}'(\phi_\delta)|dx
  + \int_{\Omega}(\phi_\delta-\overline{\phi}_{0,\delta})
  f'_{1,\delta}(\phi_\delta)dx \\
  &\le C + \int_{\Omega}(\phi_\delta-\overline{\phi}_{0,\delta})
  f'_{1,\delta}(\phi_\delta)dx,
\end{align*}
where the constant $C>0$ does not depend on $\delta$ since $0<\overline{\phi}_{0,\delta}<1$. This proves \eqref{3.f2}.

We conclude from \eqref{3.f1} and \eqref{3.f2} that
\begin{align*}
  \int_\Omega|f'_{1,\delta}(\phi_\delta)| dx
  \le C\|\na\mu_\delta\|_{L^{4/3}(\Omega)} + C.
\end{align*}
Hence, taking into account the $L^\infty(0,T;L^2(\Omega))$ bounds for $c_\delta$ and $\phi_\delta$ from \eqref{f11} and \eqref{17},
\begin{align*}
  \bigg|\int_\Omega\mu_\delta dx\bigg|
  &= \bigg|\int_\Omega\big(
  f'_{1,\delta}(\phi_\delta) - c_\delta + 1 -2\phi_\delta\big)
  dt\bigg| \le C\|\na\mu_\delta\|_{L^{4/3}(\Omega)} + C.
\end{align*}
It follows from the $L^2(0,T;L^{4/3}(\Omega))$ bound for $\na\mu_\delta$ in \eqref{18} that
\begin{align*}
  \|\overline{\mu}_\delta\|_{L^2(0,T)}
  \le C(\Omega)\bigg\|\int_\Omega\mu_\delta dx\bigg\|_{L^2(0,T)}
  \le C\|\na\mu_\delta\|_{L^2(0,T;L^{4/3}(\Omega))} + C(T) \le C.
\end{align*}
We conclude from the Poincar\'e--Wirtinger inequality that
\begin{align*}
  \|\mu_\delta\|_{L^2(0,T;L^{4/3}(\Omega))}
  &\le \|\mu-\overline{\mu}_\delta\|_{L^2(0,T;L^{4/3}(\Omega))}
  + \|\overline{\mu}_\delta\|_{L^2(0,T;L^{4/3}(\Omega))} \le C.
\end{align*}
Together with \eqref{18}, this finishes the proof.
\end{proof}

The following estimates improve \eqref{f11}.

\begin{lemma}[Estimates for $\phi_\delta$]
There exists a constant $C>0$ independent of $\delta$ such that
\begin{align}
  \|\phi_\delta\|_{L^2(0,T;D^{2s}(\Omega)}
  + \delta\|\Delta\phi_\delta\|_{L^2(\Omega_T)}\leq C.\label{f19}
\end{align}
\end{lemma}

\begin{proof}
We know that \eqref{ad18} holds a.e.\ in $\Omega_T$. Therefore, we can write
\begin{align}\label{f18}
  (-\Delta&)^s_\Omega\phi_\delta + f'_{1,\delta}(h_k(\phi_\delta))
  - \delta\Delta\phi_\delta\\
  &= \mu_\delta+c_\delta - 1 + 2\phi_\delta
  + f'_{1,\delta}(h_k(\phi_\delta)) - f'_{1,\delta}(\phi_\delta)
  \quad \mbox{a.e. in } \Omega_T,\nonumber
\end{align}
where
\begin{align*}
	h_k(z)=\begin{cases}
	-k&\mbox{for } z<-k,    \\
	z&\mbox{for } -k\leq z\leq k, \\
	k &\mbox{for }  z>k.
	\end{cases}
\end{align*}
The truncation $h_k(\phi_\delta)$ is needed to satisfy the conditions of Lemma \ref{lem.pos}. Since by construction, $f'_{1,\delta}$ is Lipschitz continuous with constant $1/\delta$, we have
\begin{align*}
  \|f'_{1,\delta}(h_k(\phi_\delta))
  - f'_{1,\delta}(\phi_\delta)\|_{L^2(\Omega)}
  \leq \frac{1}{\delta}\|h_k(\phi_\delta)-\phi_\delta\|_{L^2(\Omega)}.
\end{align*}
We claim that the right-hand side converges to zero as $k\to\infty$, for fixed $\delta>0$. Indeed, since $s>0$, there exists $p>1$ such that the embedding $D^s(\Omega)\hookrightarrow L^{2p}(\Omega)$ (with constant $C_p>0$) is continuous. Thus, using Markov's inequality,
\begin{align*}
  \|h_k(\phi_\delta)-\phi_\delta\|_{L^2(\Omega)}^2
  &= \int_{\{\phi_\delta<-k\}}|-k-\phi_\delta|^2 dx
  + \int_{\{\phi_\delta>k\}}|k-\phi_\delta|^2 dx \\
  &\le \int_{\{\phi_\delta<-k\}}|\phi_\delta|^2 dx
  + \int_{\{\phi_\delta>k\}}|\phi_\delta|^2 dx
  = \int_{\{|\phi_\delta|>k\}}|\phi_\delta|^2 dx \\
  &\le \frac{1}{k^{2(p-1)}}\int_\Omega|\phi_\delta|^{2p}dx
  \le \frac{C_p^{2p}}{k^{2(p-1)}}\|\phi_\delta\|_{D^s(\Omega)}^{2p}\to 0
  \quad\mbox{as }k\to\infty.
\end{align*}
It follows that $\|f'_{1,\delta}(h_k(\phi_\delta)) - f'_{1,\delta}(\phi_\delta)\|_{L^2(\Omega)}\to 0$ as $k\to\infty$ for any fixed $\delta>0$.

Next, we multiply \eqref{f18} by $(-\Delta)_\Omega^s\phi_\delta$ and take into account estimates \eqref{f11} and \eqref{17}:
\begin{align*}
  \int_\Omega &|(-\Delta)^s_\Omega\phi_\delta|^2dx
  + \int_\Omega f'_{1,\delta}(h_k(\phi_\delta))
  (-\Delta)^s_\Omega\phi_\delta dx
  + \delta \int_\Omega |(-\Delta)^{(1+s)/2}_\Omega\phi_\delta|^2dx \\
  &\leq \|(-\Delta)^s_\Omega \phi_\delta\|_{L^2(\Omega)}
  \big(\|\mu_\delta+c_\delta-1+2\phi_\delta\|_{L^2(\Omega)}
  + \|f'_{1,\delta}(h_k(\phi_\delta))
  -f'_{1,\delta}(\phi_\delta)\|_{L^2(\Omega)}\big)\\
  &\leq \frac{1}{2}\|(-\Delta)^s_\Omega\phi_\delta\|_{L^2(\Omega)}^2
  + C\|\mu_\delta\|_{W^{1,4/3}(\Omega)}^2 + C,
\end{align*}
where we have used the embedding $W^{1,4/3}(\Omega)\hookrightarrow L^2(\Omega)$ which holds for $d\le 4$. We know from Lemma \ref{lem.pos} that the second integral on the left-hand side is nonnegative (here, we need the truncation $h_k$). We conclude that $((-\Delta)_\Omega^s\phi_\delta)$ is bounded in $L^2(\Omega_T)$, which shows the first claim of \eqref{f19}.

We know that $\Delta\phi_\delta\in L^2(\Omega_T)$. Hence, we can multiply \eqref{ad18} by $\Delta\phi_\delta$, leading to
\begin{align*}
  \int_0^T& \int_\Omega|(-\Delta)^{(1+s)/2}_\Omega\phi_\delta|^2dxdt
  + \int_0^T\int_\Omega f''_{1,\delta}(\phi_\delta)
  |\nabla\phi_\delta|^2dxdt
  + \delta\int_0^T\int_\Omega (\Delta\phi_\delta)^2dxdt\\
  &\leq \frac{\delta}{2}\int_0^T\int_\Omega (\Delta\phi_\delta)^2dxdt
  + \frac{C}{\delta}\int_0^T\int_\Omega |\mu_\delta
  + c_\delta-1+2\phi_\delta|^2dxdt.
\end{align*}
By the $L^2(\Omega)$ bounds \eqref{f11} and \eqref{17} for $c_\delta$ and $\phi_\delta$ and the $W^{1,4/3}(\Omega)$ bound \eqref{f15} for $\mu_\delta$,
\begin{align*}
  \delta\int_0^T\int_\Omega (\Delta\phi_\delta)^2dxdt
  \le \frac{\delta}{2}\int_0^T\int_\Omega (\Delta\phi_\delta)^2dxdt
  + \frac{C}{\delta},
\end{align*}
which implies the second claim of \eqref{f19}.
\end{proof}

The following lemma gives some estimates for $f'_{1,\delta}(\phi_\delta)$ and $c_\delta$, which will be used later.

\begin{lemma}[Estimates for $f'_{1,\delta}(\phi_\delta)$ and $c_\delta$]
\label{lem.cdelta}
There exists a constant $C>0$ independent of $\delta$ such that
\begin{align}
  \|f'_{1,\delta}(\phi_\delta)\|_{L^2(\Omega_T)}
  + \|c_\delta^{3/2}\|_{L^2(0,T;W^{1,4/3}(\Omega))}
  + \|c_\delta\|_{L^{6}(0,T;L^{18/7}(\Omega))}
  &\leq C.\label{f28}
\end{align}
\end{lemma}

\begin{proof}
It follows from equation \eqref{ad18} that
\begin{align*}
  \|f'_{1,\delta}(\phi_\delta)\|_{L^2(\Omega_T)}
  &\le \|\mu_\delta\|_{L^2(\Omega_T)}
  + \|(-\Delta)_\Omega^s\phi_\delta\|_{L^2(\Omega_T)} \\
  &\phantom{xx}+ \|1-2\phi_\delta-c_\delta\|_{L^2(\Omega_T)}
  + \delta\|\Delta\phi_\delta\|_{L^2(\Omega_T)}.
\end{align*}
The right-hand side is uniformly bounded thanks to the $L^2(\Omega_T)$ bounds for $\phi_\delta$, $c_\delta$, and $\mu_\delta$, which follow from \eqref{f11}, \eqref{17}, and \eqref{f15}, respectively, and the bounds for $\phi_\delta$ from \eqref{f19}. This proves the first statement.

Next, we obtain
\begin{align*}
  \|\sqrt{c_\delta}&\na c_\delta\|_{L^2(0,T;L^{4/3}(\Omega))}
  \le \|\sqrt{c_\delta}\na(c_\delta+1-\phi_\delta)
  \|_{L^2(0,T;L^{4/3}(\Omega))}
  + \|\sqrt{c_\delta}\na\phi_\delta\|_{L^2(0,T;L^{4/3}(\Omega))} \\
  &\le C\|\sqrt{c_\delta}\na\pa_c f_\delta(\phi_\delta,c_\delta)
  \|_{L^{2}(\Omega_T)}
  + \|\sqrt{c_\delta}\|_{L^\infty(0,T;L^4(\Omega))}
  \|\na\phi_\delta\|_{L^{2}(\Omega_T)}.
\end{align*}
The right-hand side is uniformly bounded thanks to bounds \eqref{17}, \eqref{18} and bound \eqref{f19} for $\phi_\delta$. Indeed, the latter bound implies that $(\phi_\delta)$ is bounded in $L^2(0,T;D^{2s}(\Omega))$, which embeddes into $L^2(0,T;H^1(\Omega))$ if $s\ge 1/2$. It follows that $(\na c_\delta^{3/2})$ is bounded in $L^2(0,T;$ $L^{4/3}(\Omega))$. Since $(c_\delta^{3/2})$ is bounded in $L^\infty(0,T;L^{4/3}(\Omega))$ by \eqref{17}, we deduce the second bound in \eqref{f28}.

Finally, because of the continuous embedding $W^{1,4/3}(\Omega)\hookrightarrow L^{12/5}(\Omega)$ (for $d\leq 3$), the $L^2(0,T;W^{1,4/3}(\Omega))$ bound for $(c_\delta^{3/2})$ implies that $(c_\delta)$ is bounded in $L^3(0,T;L^{18/5}(\Omega))$. Hence, the interpolation inequality
\begin{align*}
  \|c_\delta\|_{L^{6}(0,T;L^{18/7}(\Omega))}
  \le C\|c_\delta\|_{L^{3}(0,T;L^{18/5}(\Omega))}^{1/2}
  \|c_\delta\|_{L^\infty(0,T;L^2(\Omega))}^{1/2}\le C
\end{align*}
concludes the proof.
\end{proof}

The final result concerns an estimate for the time derivatives.

\begin{lemma}[Estimates for the time derivatives]
There exists a constant $C$ independent of $\delta$ such that
\begin{align}
  \|\partial_tc_\delta\|_{L^{3/2}(0,T;W^{1,9}(\Omega)')}
  + \|\partial_t\phi_\delta\|_{L^2(0,T;H^1(\Omega)')}
  \leq C. \label{f20}
\end{align}
\end{lemma}

\begin{proof}
We use the bounds \eqref{16} for $\na\mu_\delta-c_\delta\na\pa_c f_\delta(\phi_\delta,c_\delta)$ in $L^2(\Omega_T)$, \eqref{18} for $\sqrt{c_\delta}\na\pa_c f_\delta(\phi_\delta,$ $c_\delta)$ in $L^2(\Omega_T)$, and \eqref{f28} for $c_\delta$ in $L^6(0,T;L^{18/7}(\Omega))$ to infer from \eqref{ad17} that for any $\psi\in L^3(0,T;W^{1,9}(\Omega))$,
\begin{align*}
  \bigg|\int_0^T&\int_\Omega\partial_t c_\delta \psi dxdt \bigg|
  \leq \|c_\delta\|_{L^{6}(0,T;L^{18/7}(\Omega))}
  \|M(\phi_\delta)\|_{L^\infty(\Omega_T)} \\
  &\phantom{xx}\times\|\nabla\psi\|_{L^3(0,T;L^{9}(\Omega))}
  \|\nabla\mu_\delta-c_\delta\nabla \partial_cf_\delta
  (\phi_\delta,c_\delta)\|_{L^2(\Omega_T)} \\
  &+ \big\|e^{-[\phi_\delta]_+^1}\big\|_{L^\infty(\Omega_T)}
  \|\sqrt{c_\delta}\|_{L^{12}(0,T;L^{36/7}(\Omega))} \\
  &\phantom{xx}\times\|\sqrt{c_\delta}\nabla\partial_c
  f_\delta(\phi_\delta,c_\delta)\|_{L^2(\Omega_T)}
  \|\nabla\psi\|_{L^{12/5}(0,T;L^{36/11}(\Omega))} \\
  &+ \delta\|\nabla c_\delta\|_{L^2(\Omega_T)}
  \|\nabla\psi\|_{L^2(\Omega_T)} \leq C\|\psi\|_{L^{3}(0,T;W^{1,9}(\Omega))}.
\end{align*}
This finishes the proof the first claim of \eqref{f20}.
Similarly, we obtain the second claim of \eqref{f20}.
\end{proof}


\subsection{Limit $\delta\to 0$}

In this section, we perform the limit $\delta\to 0$ in the weak formulation of \eqref{ad16}--\eqref{ad18}. We deduce from bounds \eqref{f11}, \eqref{f15}, and \eqref{f19} that there exist subsequences (not relabeled) such that, as $\delta\to 0$,
\begin{align*}
 \phi_\delta\rightharpoonup^*\phi &\quad\mbox{weakly* in }
  L^\infty(0,T;D^s(\Omega))\cap L^2(0,T;D^{2s}(\Omega)),
  \nonumber \\ 
  \na\phi_\delta \rightharpoonup \na\phi &\quad \mbox{weakly in }
  L^2(\Omega_T), \nonumber \\ 
  \mu_\delta\rightharpoonup \mu &\quad \mbox{weakly in }
  L^2(\Omega_T)\cap L^2(0,T;W^{1,4/3}(\Omega)).
  \nonumber 
\end{align*}
In view of the compact embedding $D^s(\Omega)\hookrightarrow L^2(\Omega)$ for $s>0$, the bounds \eqref{f11} and \eqref{f20} for $\phi_\delta$ allow us to apply the Aubin--Lions lemma to infer that, up to a subsequence,
\begin{align*}
  \phi_\delta\rightarrow \phi\quad\mbox{strongly in }
  C([0,T];L^2(\Omega)).
\end{align*}
This implies that
\begin{align*}
  M(\phi_\delta)\rightarrow M(\phi),\quad
  e^{-[\phi_\delta]^1_+}\rightarrow e^{-[\phi]^1_+}
  \quad\mbox{strongly in } L^p(\Omega_T) \mbox{ for }p\in [1,\infty).
\end{align*}

Next, we prove that $0<\phi<1$ a.e.\ in $\Omega_T$. In \cite[Lemma 5.3]{JuLi24}, this follows from the degeneracy of $M(\phi)$. Since the mobility is not degenerate in our case, we exploit the singularity as in \cite[p.~5273]{GGG17}. For the convenience of the reader, we present the full proof.

\begin{lemma}[Upper and lower bounds for $\phi$]
The limit function $\phi$ satisfies $0<\phi<1$ a.e.\ in $\Omega_T$.
\end{lemma}

\begin{proof}
Let $\eta$ be a positive constant such that $1/2\in (\eta,1-\eta)$. Define the sets
\begin{align*}
  E_\delta^\eta &:= \{(x,t)\in \Omega_T:\phi_\delta(x,t)>1-\eta
  \mbox{ or }\phi_\delta(x,t)<\eta \}, \\
  E^{\eta} &:= \{(x,t)\in \Omega_T:\phi(x,t)>1-\eta \mbox{ or }
  \phi(x,t)<\eta \}.
\end{align*}
It follows from $\phi_\delta\rightarrow \phi$ a.e.\ in $\Omega_T$ and Fatou's lemma that
\begin{align}\label{3.auxp}
  \mbox{meas}(E^\eta)\leq \liminf_{\delta\to 0^+}
  \mbox{meas}(E_\delta^\eta).
\end{align}
Recall that estimate \eqref{f28} gives a uniform bound for $f'_{1,\delta}(\phi_\delta)$ in $L^1(\Omega_T)$. Since $f'_{1,\delta}(x)\le 0$ for $x\in (0,1/2]$, $f'_{1,\delta}(x)\geq 0$ for $x\in [1/2,1)$, and $f'_{1,\delta}$ is nondecreasing, we have
\begin{align*}
  C\geq \|f'_{1,\delta}(\phi_\delta)\|_{L^1(\Omega_T)}
  \geq \int_{E_\delta^\eta}|f'_{1,\delta}(\phi_\delta)|dxdt
  \geq \min\{f'_{1,\delta}(1-\eta),-f'_{1,\delta}(\eta) \}
  \mbox{meas}(E_\delta^\eta),
\end{align*}
which implies that
\begin{align*}
  \mbox{meas}(E_\delta^\eta)
  \leq \frac{C}{\min\{f'_{1,\delta}(1-\eta),-f'_{1,\delta}(\eta) \}}.
\end{align*}
We pass to the limit $\delta\rightarrow 0$ in the previous inequality, using \eqref{3.auxp}:
\begin{align*}
 \mbox{meas}(E^\eta)
  &\leq \liminf_{\delta\rightarrow 0}\mbox{meas}(E_\delta^\eta)
  \le \frac{C}{\min\{f'_1(1-\eta),-f'_1(\eta) \}}.
\end{align*}
Finally, the limit $\eta\to 0$ yields
\begin{align*}
  \mbox{meas}\{(x,t)\in\Omega_T:\phi\geq 1 \mbox{ or } \phi\leq 0 \}=0,
\end{align*}
which completes the proof of the lemma.
\end{proof}

The previous lemma shows that $f'_1(\phi)$ is well-defined. Then, in view of the a.e.\ convergence of $\phi_\delta$, we have $f'_{1,\delta}(\phi_\delta)\to f'_1(\phi)$ a.e.\ in $\Omega_T$ as $\delta\to 0$. The $L^2(\Omega_T)$ bound of $f'_{1,\delta}(\phi_\delta)$ in \eqref{f28} implies that
\begin{align*}
  f'_{1,\delta}(\phi_\delta)\rightharpoonup f'_1(\phi)
  \quad\mbox{weakly in }L^2(\Omega_T).
\end{align*}

Next, we apply the Aubin--Lions lemma to $(c_\delta)$, taking into account the bounds in \eqref{f28} and \eqref{f20}, to obtain (up to a subsequence)
\begin{align*}
  c_\delta\rightarrow c\quad \mbox{strongly in }L^2(\Omega_T).
\end{align*}
It follows from the bounds for $c_\delta$ in \eqref{f28} that
\begin{align*}
  c_\delta\rightarrow c\quad &\mbox{strongly in }L^p(0,T;L^q(\Omega))
  \mbox{ for }p\in [1,6),\; q\in [1,18/7) \\ 
  c_\delta^{3/2}\rightarrow c^{3/2}\quad &\mbox{strongly in }
  L^p(0,T;L^q(\Omega))\mbox{ for }p\in [1,2),\; q\in [1,12/5).
\end{align*}
Therefore, for $\psi\in C_0^\infty(\Omega_T;\R^d)$,
\begin{align*}
  \int_0^T\int_\Omega & c_\delta\nabla c_\delta\cdot \psi dxdt
  = -\frac{1}{2}\int_0^T\int_\Omega c_\delta^{2}\diver\psi dxdt \\
  &\rightarrow -\frac{1}{2}\int_0^T\int_\Omega c^{2}\diver\psi dxdt
  = \int_0^T\int_\Omega c\nabla c\cdot\psi dxdt,
\end{align*}
which yields $c_\delta\na c_\delta\to c\na c$ in the sense of distributions. Actually, $\sqrt{c_\delta}\na c_\delta$ is uniformly bounded in $L^2(0,T;L^{4/3}(\Omega))$ (see the proof of Lemma \ref{lem.cdelta}) and $\sqrt{c_\delta}$ is uniformly bounded in $L^{12}(0,T;L^{36/7}(\Omega))$ (see \eqref{f28}). Hence, $(c_\delta\na c_\delta)$ is bounded in $L^{12/7}(0,T;L^{18/17}(\Omega))$, which implies that
\begin{align*}
  c_\delta\na c_\delta\rightharpoonup c\na c
  \quad\mbox{weakly in }L^{12/7}(0,T;L^{18/17}(\Omega)).
\end{align*}

The previous estimates and convergences show that
\begin{align*}
  M(\phi_\delta)\nabla\mu_\delta\rightharpoonup M(\phi)\na\mu
  &\quad\mbox{weakly in }L^1(\Omega_T), \\
  M(\phi_\delta)c_\delta\na c_\delta \rightharpoonup M(\phi)c\na c
  &\quad\mbox{weakly in }L^1(\Omega_T), \\
  M(\phi_\delta)c_\delta\na\phi_\delta\rightharpoonup M(\phi)c\na\phi
  &\quad\mbox{weakly in }L^1(\Omega_T).
\end{align*}
Hence, we deduce from bound \eqref{16} that
\begin{align*}
  M(\phi_\delta)&\big(\nabla\mu_\delta
  - c_\delta(\nabla c_\delta-\nabla\phi_\delta)\big) \\
  &\rightharpoonup M(\phi)\big(\nabla\mu-c(\nabla c-\nabla\phi)\big)
  \quad\mbox{weakly in }L^2(\Omega_T).
\end{align*}
Moreover, we have
\begin{align*}
  c_\delta M(\phi_\delta)&\big(\nabla\mu_\delta
  - c_\delta(\nabla c_\delta-\nabla\phi_\delta)\big) \\
  &\rightharpoonup cM(\phi)\big(\nabla\mu-c(\nabla c-\nabla\phi)\big)
  \quad\mbox{weakly in }L^1(\Omega_T).
\end{align*}
Estimates \eqref{f11} and \eqref{17} for $\sqrt{\delta}\na\phi_\delta$ and $\sqrt{\delta}\na c_\delta$ in $L^2(\Omega_T)$ give, for any $\psi\in C_0^\infty(\Omega_T)$,
\begin{align*}
  \delta\int_0^T\int_\Omega \Delta\phi_\delta \psi dxdt
  &\le\sqrt{\delta}\big(\sqrt{\delta}\|\na\phi_\delta\|_{L^2(\Omega_T)}
  \big)\|\na\psi\|_{L^2(\Omega_T)}
  \le C\sqrt{\delta}\rightarrow 0, \\
  \delta\int_0^T\int_\Omega \Delta c_\delta\psi dxdt
  &\le \sqrt{\delta}\big(\sqrt{\delta}\|\na c_\delta\|_{L^2(\Omega_T)}\big)
  \|\na\psi\|_{L^2(\Omega_T)}
  \le C\sqrt{\delta}\rightarrow 0.
\end{align*}
Thus, based on the previous convergence results, we can pass to the limit $\delta\to 0$ in the weak formulation of \eqref{ad16}--\eqref{ad18}
for $(\phi_\delta,c_\delta,\mu_\delta)$ to deduce that the triplet $(\phi,c,\mu)$ is a weak solution to \eqref{m1}--\eqref{m3}. This finishes the proof of Theorem \ref{main}.


\end{document}